\documentclass[11pt]{article}
\usepackage{amsfonts,amsmath,amssymb,amsthm}
\usepackage{hyperref}
\usepackage{kbordermatrix}

\usepackage{epsfig}
\usepackage{epstopdf}
\usepackage{blkarray}
\usepackage{kbordermatrix}
\usepackage{tikz}
\usetikzlibrary{arrows,decorations.pathmorphing,backgrounds,positioning,fit,petri}
\tikzset{nodde/.style={circle,draw=blue!50,fill=blue!20,inner sep=1.2pt}}
\tikzset{nodeblack/.style={circle,draw=black,fill=black!30,inner sep=1.2pt}}

\newtheorem{theorem}{Theorem}  [section]

\newtheorem{lemma}[theorem]{Lemma}
\newtheorem{claim}[theorem]{Claim}

\newtheorem{conj}[theorem]{Conjecture}
\theoremstyle{definition}

\newcommand{\bproof}{\begin{proof}}
\newcommand{\eproof}{\end{proof}}

\newcommand{\R}{{\mathbb R}}
\newcommand{\rat}{{\mathbb Q}}

\newcommand{\rank}{{\mbox{rank }}}
\newcommand{\sm}{\setminus}

\newcommand{\Hc}{{\mathcal H}}
\newcommand{\Lc}{{\mathcal L}}
\newcommand{\Hco}{\overline{{\mathcal H}}}

\newcommand{\Rc}{{\mathcal R}}
\newcommand{\Xc}{{\mathcal X}}
\DeclareMathOperator{\val}{val}
\DeclareMathOperator{\cov}{cov}

\title{Coincident Rigidity of 2-Dimensional Frameworks
}

\author{Hakan Guler\thanks{Department of Mathematics, Faculty of Arts \& Sciences, Kastamonu University, Kastamonu, Turkey.
Email: hakanguler19@gmail.com}\ \ and Bill Jackson\thanks{School of Mathematical Sciences, Queen Mary
University of London, Mile End Road, London E1 4NS, UK. 
Email: b.jackson@qmul.ac.uk}}

\begin{document}
\maketitle

\begin{abstract}
Fekete, Jord\'an and Kaszanitzky \cite{FJK} characterised the graphs which can be realised  as  2-dimensional, infinitesimally rigid, bar-joint frameworks in which two given vertices are coincident.  We formulate a conjecture which would extend their characterisation to an arbitrary set $T$ of vertices and verify our conjecture when $|T|=3$.
\end{abstract}
\maketitle
\section{Introduction}
A $d$-dimensional (bar-and-joint) {\em framework} is a pair $(G,p)$ where $G=(V,E)$ is
a simple graph and $p:V\rightarrow \R^d$ is a map, which we refer to as the {\em realisation} of the framework. The {\em length} of each edge of $(G,p)$ is given by the Euclidean distance betwean its end points. The framework is said to be {\em rigid} if every continuous motion of the vertices which preserves the lengths of the edges, preserves the distance between all pairs of vertices.  It is {\em infinitesimally rigid} if it satisfies the stronger property that every infinitesimal motion of the vertices which preserves the lengths of the edges is an infinitesimal isometry of $\R^d$.

It is not difficult to see that a 1-dimensional framework is rigid if and only if its underlying graph is connected, but for $d\geq 2$, the decision problem of deciding whether a given $d$-dimensional framework is rigid is NP-hard by a result of Abbot \cite{A}. This  problem becomes more tractable, however, if we restrict our attention to {\em generic frameworks} 
i.e.~framworks $(G,p)$ for which the set of the coordinates $p(v)$, $v\in V$, is algebraically
independent over the rational numbers. Asimow and Roth \cite{AR} showed that the properties of rigidty and infinitesimal rigidity are equivalent for such frameworks and depend only on the underlying graph. This allows us to define a graph $G$ as being {\em rigid} in $\R^d$ if some, or equivalently every, generic realisation of $G$ in $\R^d$ is rigid.

Graphs which are rigid in $\R^d$ have been characterised for $d=1,2$: we have already seen that $G$ is rigid in $\R$ if and only if $G$ is connected and a fundamental result of Pollaczek-Geiringer \cite{PG}, subsequently rediscovered by Laman \cite{L}, characterises when $G$
is rigid in $\R^2$. Finding a characterisation when $d\geq 3$ is the main open problem in distance geometry, although characterisations do exist for certian families of graphs. A common technique used to show that a family of graphs $G$ is rigid in $\R^d$ is to reduce $G$ to a smaller graph $G'$ in the family  by some operation, apply induction  to deduce that $G'$ is rigid, and then show
that the inverse operation preserves rigidity. The last step in this proof strategy frequently uses a geometric argument based on a nongeneric realisation of $G$. More precisely, we extend a generic (and hence infinitesimally rigid) realisation $p'$ of $G'$ to a realisation $p$ of $G$ by choosing special positions for the vertices of $V(G)\sm V(G')$ which make it easy to conclude that $(G,p)$ is also infinitesimally rigid, then use the fact (see Section \ref{sec:prelim}) that if some realsation of $(G,p)$ is infinitesimally rigid then every generic realsation of $(G,p)$ is infinitesimally rigid. 

This approach has stimulated interest in such special position frameworks. Jackson and Jord\'an \cite{JJ} characterised when a graph $G$ has an infinitesimally rigid
realisation in $\R^2$ in which three given vertices are collinear. Another result due to 
Fekete, Jord\'an and Kaszanitzky \cite{FJK}, which is closer to our interests in this paper, characterises when a graph $G$ has an infinitesimally rigid
realisation in $\R^2$ in which  two given vertices are coincident. We need some new notation to describe their theorem. Given a graph $G=(V,E)$ and $T\subseteq V$ we use $G/T$ to denote the graph obtained by contracting the vertices in $T$ to a single vertex. When $T=\{u,v\}$ we often use $G/uv$ instead of $G/T$. We also use $G-uv$ to denote the graph obtained from $G$ by deleting the edge $uv$ if it exists in $E$ (and putting $G-uv=G$ when $uv\not\in E$). 

\begin{theorem}\cite{FJK}\label{C3-thm:uv-FJK}
Let $G$ be a graph and $u,v$ be distinct vertices of $G$. Then
$G$ has an infinitesimally rigid, $\{u,v\}$-coincident realisation in $\mathbb{R}^2$ if and only if $G-uv$ and $G/{uv}$ are
both rigid in $\mathbb{R}^2$.
\end{theorem}

We will obtain an analogous characterisation for three coincident vertices.

\begin{theorem}\label{C3-thm:int_main1}
Let $G=(V,E)$ be a graph, $u,v,w$ be distinct vertices of $G$,
and $G'=G-uv-uw-vw$.
Then $G$ has an infinitesimally rigid, $\{u,v,w\}$-coincident realisation in $\mathbb{R}^2$ if and only if $G'$ 
and $G'/S$ are rigid in $\R^2$ for all $S\subseteq \{u,v,w\}$ with $|S|\geq 2$.
\end{theorem}


\section{Preliminaries}\label{sec:prelim}
The {\em rigidity matrix} $R(G,p)$ of a $d$-dimensional
framework $(G,p)$ is the matrix of size $|E|\times d|V|$, where in the row corresponding to an edge
$uv\in E$, the entries in the columns corresponding to $u$ and $v$ are $p(u)-p(v)$ and $p(v)-p(u)$, respectively
and all other entries are zeros. The right kernel of $R(G,p)$ is the space of {\em infinitesimal motions of} $(G,p)$. This space has dimension at least 
${d+1\choose 2}$ when $|V|\geq d$ since every infinitesimal isometry of $\mathbb{R}^d$ is an infinitesimal motion of $(G,p)$.
This implies that $\rank R(G,p)\leq d|V|-{d+1\choose 2}$ when $|V|\geq d$, and $(G,p)$ is {\em infinitesimally rigid}
if $\rank R(G,p)$ acheives this upper bound when $|V|\geq d$. For the case when $|V|< d$, $(G,p)$ is {\em infinitesimally rigid}
if $\rank R(G,p)={|V|\choose 2}$.
%
%
Since the rank of $R(G,p)$ will be maximised whenever $(G,p)$ is generic, this immediately implies that the infinitesimal  rigidity of a generic framework $(G,p)$ in $\R^d$ depends only on the underlying graph $G$.

The framework $(G,p)$ is said to be {\em independent} if the rows of its rigidity matrix $R(G,p)$
are linearly independent. The upper bound on $\rank R(G,p)$ gives the following necessary condition for independence which was first observed by Maxwell \cite{M}.
\begin{eqnarray}
&\mbox{If $(G,p)$ is an independent $d$-dimensional framework and $G'=(V',E')$}\nonumber\\
&\mbox{is a subgraph of $G$ with $|V'|\geq d$ then $|E'|\leq d|V'|-{d+1\choose 2}$.}\label{eq:ind}
\end{eqnarray}

 We can construct a matroid $\mathcal{R}(G,p)$ on $E(G)$ from $R(G,p)$ by defining a set $F\subseteq E(G)$ to be  independent in $\mathcal{R}(G,p)$ if the corresponding rows of ${R}(G,p)$ are linearly independent.
%
The  {\em $d$-dimensional rigidity matroid} $\mathcal{R}_d(G)$ of the graph $G$ is the matroid
$\mathcal{R}(G,p)$ for any generic $d$-dimensional framework $(G,p)$. 
The necessary condition for independence in $\mathcal{R}_d(G)$ given by (\ref{eq:ind}) is also sufficient when $d=1,2$. When $d=1$ it is equivalent to saying that $G$ is a forest. When $d=2$ it is implied by the above mentioned  characterisations of generic rigidity in  $\R^2$ due to Pollaczek-Geiringer \cite{PG}. Lov\'asz and Yemini \cite{LY} used the characterisation of independence in $\mathcal{R}_2(G)$ to determine its rank function.  
We need to introduce some new terminology to describe their result. Given a graph $G=(V,E)$ and $E'\subset E$, a {\em cover of $E'$} is a family $\mathcal X$ of subsets of $V$ such that each member of $\Xc$ has cardinality at least two and each edge in $E'$ is induced by some member of $\Xc$. The cover $\Xc$ is {\em 1-thin} if $|X_i\cap X_j|\leq 1$ for all distinct $X_i,X_j\in \Xc$.

\begin{theorem}\label{thm:LY} Let $G=(V,E)$ be a graph and $E'\subseteq E$. Then the rank of $E'$ in $\mathcal{R}_2(G)$ is given by 
$r(E')=\min\{\sum_{X\in \Xc}(2|X|-3):\mbox{$\Xc$ is a $1$-thin cover of $E'$}\}.$
\end{theorem}

For $T\subseteq V(G)$, we can define the   {\em $T$-coincident, $d$-dimensional rigidity matroid} $\mathcal{R}_{d,T}(G)$ of $G$ in the same way as the $d$-dimensional rigidity matroid. We first choose a reference vertex $t\in T$. We say that a realisation $p$ of $G$ is {\em $T$-coincident} if $p(v)=p(t)$ for all $v\in T$ and is a {\em generic $T$-coincident realisation} if $p|_{V(G)\sm(T-t)}$ is algebraically independent over $\rat$. Then $\rank R(G,p)$ will be maximised over all $T$-coincident realisations in $\R^d$ whenever $(G,p)$ is a generic $T$-coincident realisation and hence the infinitesimal  rigidity of a generic $T$-coincident framework $(G,p)$ in $\R^d$ depends only on the graph $G$ and the set $T$. We say that $G$ is {\em $T$-coincident rigid in $\R^d$} if some, or equivalently every, generic $T$-coincident realisation of $G$ in $\R^d$ is infinitesimally rigid. The  {\em $d$-dimensional $T$-coincident  rigidity matroid} $\mathcal{R}_{d,T}(G)$ of the pair $(G,T)$ is the matroid
$\mathcal{R}(G,p)$ for any generic $d$-dimensional $T$-coincident framework $(G,p)$. It is easy to see that $\mathcal{R}_{d,T}(G)=\mathcal{R}_{d}(G)$ when $|T|=1$ and that $\mathcal{R}_{1,T}(G)=\mathcal{R}_{1}(G-E_G(T))$, where $E_G(T)$ is the set of edges of $G$ induced by $T$. The results of \cite{FJK} characterise  $\mathcal{R}_{d,T}(G)$ when $d=|T|=2$. We will extend this characterisation to the case when $|T|=3$ and formulate a conjecture which would characterise  $\mathcal{R}_{2,T}(G)$ for all $T$. 

We will need the following observation which relates the $T$-coincident rigidity matroid $\mathcal{R}_{d,T}(G)$ to the rigidity matroid $\mathcal{R}_{d}(G/T)$.
Let $z_T$ be the vertex of $G/T$ corresponding to $T$.
 Given a framework $(G/T,p_T)$ in $\R^d$,
we can obtain a $T$-coincident realisation $(G,p)$ of $G$ by putting
$p(x)=p_T(z_T)$ if $x\in T$ and $p(x)=p_T(x)$ if $x\notin T$.
Furthermore, the map $q_T\mapsto q$ given by $q(x)=q_T(z_T)$ if $x\in T$ and $q(x)=q_T(x)$
if $x\notin T$ is an injective linear transformation from  $\ker R(G_T,p_T)$
to $\ker R(G,p)$.
This gives $\dim \ker R(G,p)\geq \dim \ker R(G/T,p_T)$ and hence 
\begin{equation}\label{eq:dimKer}
\rank  R(G,p)\leq \rank R(G/T,p_T) +d(|T|-1).
\end{equation}

We close this section with some graph theoretic notation and terminology. Let $G=(V,E)$ be a graph. For $X\subseteq V$, let $G[X]$ denote the
subgraph induced by $X$. Let $E_G(X)$ be the set and $i_G(X)$ be
the number of edges of $G[X]$. For a family $\mathcal{X}$
of subsets of $V$,  we put
$E_G(\mathcal{X})=\bigcup_{X\in\Xc}{E_G(X)}$ and
$i_G(\mathcal{X})=|E_G(\mathcal{X})|$. We also define
$\cov(\mathcal{X})=\{xy:\{x,y\}\subseteq X, \text{ for some } X\in\Xc\}$ and
say that $\mathcal{X}$ {\em covers}  a set $E'\subseteq E$ if $E'\subseteq \cov(\mathcal{X})$.
The {\em degree}
of a vertex $v$ in $G$ is denoted by $d_G(v)$ and the set of all  neighbours of $v$ in $G$
is denoted by $N_G(v)$. We will omit the subscripts referring to $G$ when the
graph is clear from the context.

\section{\boldmath A matroid construction}\label{sec:matroidrank}

We use a similar strategy to \cite{FJK} to characterise $\Rc_{2,T}(G)$ when $|T|=3$ and prove Theorem \ref{C3-thm:int_main1}.
Suppose $G=(V,E)$ is a graph  and $T\subseteq V$. In this section we derive necessary conditions for independence in ${\mathcal R}_{2,T}(G)$, and show that these necessary conditions for independence define 
a matroid  ${\mathcal M}_{T}(G)$ on $E(G)$ when $|T|\leq 3$. We show in the next section  that ${\mathcal  M}_T$ is equal to ${\mathcal R}_{2,T}(G)$
and then use our formula for the rank function of ${\mathcal M}_{T}$ to verify Theorem \ref{C3-thm:int_main1}.




For a fixed nonempty set $T\subseteq V$, 
we define the $T$-{\em value} of an arbitrary set $X\subseteq V$ 
by
$$
\val_T(X)=\left\{
\begin{array}{ll}
2|X|-3  &\mbox{if $X\not\subseteq T$}\\
$0$ &\mbox{if $X\subseteq T$}
\end{array}
\right.
$$
Note that $
\val_T(X)\geq 0$ whenever $|X|\geq 2$. 

We say that a non-empty family $\mathcal{H}=\{ H_1,\ldots,H_k\}$  of subsets of $ V$  is $T$-{\em compatible} if $T$ is a proper subset of $H_i$
for all $1\leq i \leq k$, and define 
its {\em $T$-value} to be 
\[
\val_T(\mathcal{H})=\sum_{i=1}^{k}{(2|H_i\setminus T|-1)}+2(|T|-1).
\]
Note that $
\val_T(\mathcal{H})\geq 0$ since  $|H_i|> |T|\geq 1$ for all $1\leq i\leq k$.

The graph $G$ is said to be $T$-{\em sparse} if
\begin{itemize}
\item $i_G(X)\leq \val_T(X)$  for all
$X\subseteq V$
with $|X|\geq 2$ and 
\item 
$i_G(\Hc)\leq \val_T(\Hc)$  for all $T$-compatible
families $\Hc$.
\end{itemize}
In particular, if $G$
is $T$-sparse, then $i_G(X)\leq 2|X|-3$ for all $X\subseteq V$ with $|X|\geq 2$ so $E$ is independent in $\mathcal{R}_2(G)$ by \cite{PG}.

We motivate these definitions by showing that  $T$-sparsity is a necessary condition for independence in the 2-dimensional $T$-coincident rigidity matroid $\mathcal{R}_{2,T}(G)$.

\begin{lemma}\label{C3-lem:geom_imp_count}
Let $G=(V,E)$ be a graph and let $T\subseteq V$ with $|T|\geq 1$.
Suppose $E$ is independent in $\mathcal{R}_{2,T}(G)$. 
Then $G$ is $T$-sparse.
\end{lemma}
\bproof 
Let $(G,p)$ be a generic $T$-coincident realisation of $G$ in $\R^2$. Then 
$R(G,p)$ has linearly independent rows. This implies that $R(G',p|_{V'})$ has independent rows for any subgraph $G'=(V',E')$ of $G$. 

Choose $X\subseteq V$ with $|X|\geq 2$ and let $J=G[X]$. Then $(J,p|_{X})$ is independent and  (\ref{eq:ind}) gives $i_G(X)=|E(J)|\leq 2|X|-3$. In addition, if $X\subseteq T$ then $i_G(X)=0$, since any edge $ab\in E(G)$ would give rise to a zero row in $R(J,p|_{X})$.  Hence $i_G(X)\leq \val_T(X)$.


Now suppose that $|T|\geq 2$ and let $\Hc=\{H_1,\ldots,H_k\}$ be a $T$-compatible family.
Consider the subgraphs of $G$ given by $L_i=G[H_i]$, $1\leq i\leq k$, and  $L=\bigcup_{i=1}^{k}L_i$.
Let $L_i/T$, respectively $L/T$, be obtained from $L_i$, respectively $L$, by
contracting $T$ into a single vertex $z_T$. Let $(L/T,p_T)$ be the realisation of $L/T$ with $p_T(x)=p(x)$ for $x\neq z_T$ and $p_T(z_T)=p(z)$ for any $z\in T$, and let $p_i$ be the restriction of $p_T$ to the vertices of $L_i/T$.
Then every edge of $L/T$ belongs to one of the subgraphs $L_i/T$
and we can use \ref{eq:ind} to obtain 
$$\rank R(L/{T},p_T)\leq \sum_{i=1}^{k}\rank R(L_i/{T},p_i)\leq \sum_{i=1}^{k}(2|V(L_i/T)|)-3)=\sum_{i=1}^{k}(2|H_i\sm T|-1).$$
We can combine this  bound with (\ref{eq:dimKer}) and the fact that $(J,p|_{X})$ is independent to obtain
$$
i_G(\Hc)=|E(J)|=\rank R(J,p|_{X})\leq
 \sum_{i=1}^{k}(2|H_i\sm T|-1)+2(|T|-1)=\val_T(\Hc).
$$
\eproof

The converse of Lemma \ref{C3-lem:geom_imp_count} holds for $|T|=1,2$. When $|T|=1$, $G$ is $T$-sparse if and only if $i_G(X)\leq 2|X|-3$ for all $X\subseteq V$ with $|X|\geq 2$ and this condition characterises independence in $\mathcal{R}_2(G)$ by \cite{PG}.  
When $|T|=2$, the condition that $G$ is $T$-sparse characterises independence in $\mathcal{R}_{2,T}(G)$ by \cite{FJK}.
When $|T|\geq 3$ we  we need a stronger condition which follows from the fact that an infinitesimally rigid $T$-coincident realisation of $G$ is an 
infinitesimally rigid $S$-coincident realisation of $G$ for all $S\subseteq T$.
Combined with Lemma \ref{C3-lem:geom_imp_count}, this implies that  we need $G$ to be $S$-sparse for all $S\subseteq T$ with $|S|\geq 2$, for $E$ to be independent in $\mathcal{R}_{2,T}(G)$. (Note that we do not need to include the case when $|S|=1$ since this follows immediately from the condition that $G$ is $T$-sparse.) 
%
We will show that this strengthened condition characterises independence in $\mathcal{R}_{2,T}(G)$ when $|T|=3$. We first obtain some preliminary results on $S$-compatible families.

\subsection{\boldmath $S$-compatible families}

\begin{lemma}\label{C3-lem:Hi-intersection}
Let $G=(V,E)$ be a graph, $S\subseteq V$ with $|S|\geq 2$ and $\Hc=\{H_1,\ldots,H_k\}$ be an $S$-compatible family in $G$.
Suppose $|H_i\cap H_j|\geq |S|+1$ for some pair $1\leq i<j\leq k$.
Then there exists an $S$-compatible family $\overline{\Hc}$ with
$\cov(\Hc)\subseteq \cov(\overline{\Hc})$ and
$\val_S(\overline{\Hc})\leq \val_S(\Hc)-1$.
\end{lemma}
\bproof We may assume that $i=k-1$ and $j=k$. Let $\Hco=\{H_1,\ldots,H_{k-2},\overline{H}_{k-1}\}$ where $\overline{H}_{k-1}= H_{k-1}\cup H_k$.
Then we have $\cov(\Hc)\subseteq \cov(\Hco)$ and
\begin{align*}
\val_S(\Hc)&=\sum_{l=1}^{k}{(2|H_l\setminus S|-1)}+2(|S|-1)\\
&=  \sum_{l=1}^{k-2}{(2|H_l\setminus S|-1)}+2(|S|-1)+(2|H_{k-1}\setminus S|-1)+(2|H_k\setminus S|-1)\\
&=  \sum_{l=1}^{k-2}{(2|H_l\setminus S|-1)}+2(|S|-1)+(2|(H_{k-1}\cup H_k)\setminus S|-1)
\\&\indent
+(2|(H_{k-1}\cap H_k)\setminus S|-1)\\
&\geq \val_S(\Hco)+1.
\end{align*}
\eproof

Given an $S$-sparse graph $G=(V,E)$ with $S\subseteq V$, we define a set $H\subseteq V(G)$ with $|H|\geq 2$ to be $S$-{\em tight}
if $i_G(H)=\val_S(H)$. 
Similarly, an $S$-compatible
family $\Hc$ is said to be $S$-{\em tight} 
if $i_G(\Hc)=\val_S(\Hc)$.

\begin{lemma}\label{C3-lem:Hi_Y-int1}
Let $G=(V,E)$ be a graph, $S\subseteq V$ with $|S|\geq 2$ and $\Hc=\{ H_1,\ldots,H_k\}$ be an $S$-compatible family
such that $H_i\cap H_j=S$
for all $1\leq i<j\leq k$. Suppose $Y\subseteq V$ 
with $|Y\cap S|\leq 1$
and $|Y\cap H_i|\geq 2$ for some $1\leq i\leq k$.
Then there exists an $S$-compatible family $\Hco$ with
$\cov(\Hc)\cup \cov(Y)\subseteq \cov(\Hco)$ for which $\val_S(\Hco)\leq \val_S(\Hc)+\val_S(Y)$.
Furthermore, if $G$ is $S$-sparse and $\Hc$ and $Y$ are both $S$-tight,
then $\Hco$ is also $S$-tight.
\end{lemma}
\bproof By reordering the elements of $\Hc$ if necessary,
we may assume that $|Y\cap H_i|\leq 1$ for all $1\leq i\leq j-1$ and $|Y\cap H_i|\geq 2$ for all $j\leq i\leq k$, for some $1\leq j\leq k$.
Let $X=Y\cup \bigcup_{i=j}^{k}{H_i}$
and $\Hco=\{H_1,\ldots,H_{j-1},X\}$. Then we have $\cov(\Hc)\cup \cov(Y)\subseteq \cov(\Hco)$, and
\begin{align*}
&\val_S(\Hc)+\val_S(Y)=\sum_{i=1}^{k}{(2|H_i\setminus S|-1)}+2(|S|-1)+(2|Y|-3)\\
&=\sum_{i=1}^{j-1}{(2|H_i\setminus S|-1)}+2(|S|-1)+\sum_{i=j}^{k}{(2|H_i\setminus S|-1)}+(2|Y|-3)\\
&=\sum_{i=1}^{j-1}{(2|H_i\setminus S|-1)}+2(|S|-1)+(2|X\setminus S|-1)\\
&\indent+2|Y\cap S|-(k-j)+2\sum_{i=j}^{k}{|Y\cap (H_i\setminus S)|}-3\\
&=\sum_{i=1}^{j-1}{(2|H_i\setminus S|-1)}+(2|X\setminus S|-1)+2(|S|-1)\\
&\indent+2|Y\cap S|-(k-j)+2\sum_{i=j}^{k}{|Y\cap H_i|}-2\sum_{i=j}^{k}{|Y\cap S|}-3\\
&=\sum_{i=1}^{j-1}{(2|H_i\setminus S|-1)}+(2|X\setminus S|-1)+2(|S|-1)\\
&\indent-(k-j)+2\sum_{i=j}^{k}{|Y\cap (H_i)|}-2|Y\cap S|(k-j)-3\\
&\geq \sum_{i=1}^{j-1}{(2|H_i\setminus S|-1)}+(2|X\setminus S|-1)+2(|S|-1)+2\sum_{i=j}^{k}{|Y\cap H_i|}-3(k-j+1)\\
&=\sum_{i=1}^{j-1}{(2|H_i\setminus S|-1)}+(2|X\setminus S|-1)+2(|S|-1)+\sum_{i=j}^{k}{(2|Y\cap H_i|-3)}\\
&= \val_S(\Hco)+\sum_{i=j}^{k}{\val_S(Y\cap H_i)}
\end{align*}
where for the inequality step we use $|Y\cap S|\leq 1$.
Since  $\val_S(Y\cap H_i)\geq 0$ for all $j\leq i\leq k$ this gives $\val_S(\Hc)+\val_S(Y)\geq \val_S(\Hco)$.

Now suppose that $G$ is $S$-sparse, and $\Hc$ and $Y$ are $S$-tight. Then we have
\[
i(\Hco)+\sum_{i=j}^{k}{i(Y\cap H_i)}\geq i(\Hc)+i(Y)=\val_S(\Hc)+\val_S(Y)
\]
\[
\geq \val_S(\Hco)+\sum_{i=j}^{k}{\val_S(Y\cap H_i)}\geq i(\Hco)+\sum_{i=j}^{k}{i(Y\cap H_i)},
\]
where the first inequality follows since the edges spanned by $\Hc$ or $Y$
are spanned by $\Hco$ and if some edge is spanned by both $\Hc$ and $Y$, then it is
spanned by $Y\cap H_i$ for some $i$. The first equality holds because $\Hc$ and $Y$
are $S$-tight, and the second inequality holds by our calculations above. The last
inequality holds because $G$ is $S$-sparse. Hence equality must hold everywhere,
which implies that $\Hco$ is also $S$-tight.
\eproof

\begin{lemma}\label{C3-lem:Hi_Y-int0}
Let $G=(V,E)$ be a graph, $S\subseteq V$ with $|S|\geq 2$ and $\Hc =\{H_1,\ldots,H_k\}$ be an $S$-compatible family
such that $H_i\cap H_j=S$ for
all $1\leq i<j\leq k$. Suppose that $Y\subseteq V$ with $Y\cap S=\emptyset$, $|Y\cap H_i|\leq 1$ for all $1\leq i\leq k$ and
$|Y\cap H_a|=|Y\cap H_b|=1$ for some $(a,b)$ with  $1\leq a<b\leq k$.
Then there is an $S$-compatible family $\Hco$ with
$\cov(\Hc)\cup \cov(Y)\subseteq \cov(\Hco)$ for which $\val_S(\Hco)=\val_S(\Hc)+\val_S(Y)$.
Furthermore, if $G$ is $S$-sparse and $\Hc$ and $Y$ are both $S$-tight, then
$\Hco$ is also $S$-tight.
\end{lemma}
\bproof
We may assume that $a=k-1$ and $b=k$.
Let $\Hco=\{H_1,\ldots,H_{k-2},\overline{H}_{k-1}\}$ where $\overline{H}_{k-1}=H_{k-1}\cup H_k\cup Y$. Then $\cov(\Hc)\cup \cov(Y)\subseteq \cov(\Hco)$ and we have
\begin{align*}
&\val_S(\Hc)+\val_S(Y)=\sum_{i=1}^{k}{(2|H_i\setminus S|-1)}+2(|S|-1)+(2|Y|-3)\\
&=\sum_{i=1}^{k-2}{(2|H_i\setminus S|-1)}+2(|S|-1)+(2|H_{k-1}\setminus S|-1)+(2|H_k\setminus S|-1)+(2|Y|-3)\\
&=\sum_{i=1}^{k-2}{(2|H_i\setminus S|-1)}+2(|S|-1)+(2(|H_{k-1}\setminus S|+|H_k\setminus S|+|Y|)-1)-4\\
&=\sum_{i=1}^{k-2}{(2|H_i\setminus S|-1)}+(2|(H_{k-1}\cup H_k\cup Y)\setminus S|-1)+2(|S|-1)\\
&\indent+2|Y\cap (H_{k-1}\setminus S)|+2|Y\cap (H_k\setminus S)|-4\\
&=\val_S(\Hco).
\end{align*}

Now suppose that $G$ is $S$-sparse and $\Hc$ and $Y$ are $S$-tight. Then we have
\[
i(\Hc)+i(Y)=\val_S(\Hc)+\val_S(Y)=\val_S(\Hco)\geq i(\Hco)\geq i(\Hc)+i(Y)
\]
where the last inequality follows since $|Y\cap H_i|\leq 1$ for all $1\leq i\leq k$.
Hence equality must hold everywhere which implies that $\Hco$ is also $S$-tight.
\eproof

\begin{lemma}\label{C3-lem:tight_set_int}
Let $G=(V,E)$ be an $S$-sparse graph for some $S\subseteq V$ with $|S|\geq 2$. Suppose
$X,Y\subseteq V$ are $S$-tight sets in $G$ with $|X\cap Y|\geq 2$ and $X,Y\not\subseteq S$. Then $X\cap Y\not\subseteq S$, and $X\cup Y$ and $X\cap Y$ are both $S$-tight. 
\end{lemma}
\bproof
First note that, since $G$ is $S$-sparse, we have
\begin{align*}
2|X|-3+2|Y|-3=\val_S(X)+\val_S(Y)
&=i(X)+i(Y)\\
                                 &\leq i(X\cap Y)+i(X\cup Y)\\
								 &\leq \val_S(X\cap Y)+\val_S(X\cup Y)\\
								 &=\val_S(X\cap Y)+2|X\cup Y|-3.
\end{align*}
This implies that $\val_S(X\cap Y)\geq 1$ and hence $X\cap Y\not\subseteq S$.
%
This gives
$\val_S(X\cap Y)=2|X\cap Y|-3$ and hence equality must hold throughout the above sequence of (in)equalities. In particular,
$\val_S(X\cup Y)=i(X\cup Y)$ and $\val_S(X\cap Y)=i(X\cap Y)$, so
$X\cup Y$ and $X\cap Y$ are $S$-tight.
\eproof

Given a graph $G=(V,E)$  and $T\subseteq V$ with $|T|\geq 1$, we say  that $G$ is {\em strongly $T$-sparse} if $G$ is $S$-sparse for all $\emptyset\neq S\subseteq T$. 
\begin{lemma}\label{C3-lem:max_cover}
Suppose $\Hc_i$ is an $S_i$-compatible family in a graph $G=(V,E)$ for $i=1,2$ and $S_1\cap S_2\neq \emptyset$. Then there exists an $(S_1\cup S_2)$-compatible family in $G$ with $\cov(\Hc_1)\cup \cov(\Hc_2)\subseteq \cov(\Hc)$. Furthermore, if $G$ is strongly $(S_1\cup S_2)$-sparse and $\Hc_i$ is $S_i$-tight for $1\leq i\leq 2$, then $\Hc$ is $(S_1\cup S_2)$-tight.
\end{lemma}
\bproof 
Let $\Hc_1=\{H_1,\ldots,H_k\}$ and $\Hc_2=\{\overline{H}_1,\ldots,\overline{H}_l\}$.
Let $\mathcal{G}=(\mathcal{V},\mathcal{E})$ be the bipartite
graph with vertex bipartition $(\Hc_1,\Hc_2)$,
and edge set
\[
\mathcal{E}:=\{H_i\overline{H}_j:|(H_i\setminus S_1)\cap (\overline{H}_j\setminus S_2)|\geq 1, 1\leq i\leq k, 1\leq j\leq l\}.
\]
Let $\mathcal{G}_i=(\mathcal{V}_i,\mathcal{F}_i)$, $1\leq i\leq r$, be the connected components of $\mathcal{G}$ and
put $V_i=\bigcup_{H\in \mathcal{V}_i}{H}$ for  $1\leq i\leq r$. Let
\[
\Hc_{\text{union}}=\{V_i\cup S_1\cup S_2:1\leq i\leq r\} \mbox{ and }
\Hc_{\text{int}}=\{H_i\cap \overline{H}_j:H_i\overline{H}_j\in \mathcal{E}\}.
\]
Then  $\Hc_{\text{int}}$ is $(S_1\cap S_2)$-compatible, $\Hc_{\text{union}}$ is $(S_1\cup S_2)$-compatible and we have
$\cov(\Hc_1)\cup \cov(\Hc_2)\subseteq \cov(\Hc_{\text{union}})$.

Now suppose that $G$ is strongly $(S_1\cup S_2)$-sparse and $\Hc_i$ is $S_i$-tight for $1\leq i\leq 2$.
We will show that $\Hc_{\text{union}}$ is $(S_1\cup S_2)$-tight to complete the proof.
Every edge in $E$ which is covered by 
either $\Hc_1$ or $\Hc_2$ is covered by $\Hc_{\text{union}}$ and every edge
covered by both $\Hc_1$ and $\Hc_2$ is covered by $\Hc_{\text{int}}$. This implies that
$i(\Hc_1)+i(\Hc_2)\leq i(\Hc_{\text{union}})+i(\Hc_{\text{int}})$.
Since $|\mathcal{V}|=k+l$ and $r$ is the number of connected components of $\mathcal{G}$,
\begin{equation}\label{eq:rEkl1}
r+|\mathcal{E}|\geq k+l.
\end{equation}
We also have
\begin{equation}\label{eq:double-count1}
\begin{aligned}
&\sum_{i=1}^{r}{(|V_i\cup S_1\cup S_2|-|S_1\cup S_2|)}+\sum_{H_i\overline{H}_j\in \mathcal{E}}{(|H_i\cap \overline{H}_j|-|S_1\cap S_2|)}\\
=&\sum_{i=1}^{k}{(|H_i|-|S_1|)}+\sum_{i=1}^{l}{(|\overline{H}_i|-|S_2|)}
\end{aligned}
\end{equation}
as a vertex $x\notin S_1\cup S_2$ contributes the same amount (one or two) to both sides of (\ref{eq:double-count1}),
and a vertex $s\in S_1\cup S_2$ contributes zero to both sides of (\ref{eq:double-count1}).

Then since $\Hc_{\text{int}}$ is  $(S_1\cap S_2)$-compatible and $\Hc_{\text{union}}$ is $(S_1\cup S_2)$-compatible, we have
\begin{align*}
&\sum_{i=1}^{k}{(2|H_i\setminus S_1|-1)}+2(|S_1|-1)+\sum_{i=1}^{l}{(2|\overline{H}_i\setminus S_2|-1)}+2(|S_2|-1)\\
&=\val_{S_1}(\Hc_1)+\val_{S_2}(\Hc_2)\\
&=i(\Hc_1)+i(\Hc_2)\\
&\leq i(\Hc_{\text{union}})+i(\Hc_{\text{int}})\\
&\leq \val_{S_1\cup S_2}(\Hc_{\text{union}})+\val_{S_1\cap S_2}(\Hc_{\text{int}})\\
&= \sum_{i=1}^{r}{(2|(V_i\cup S_1 \cup S_2)\setminus (S_1\cup S_2)|-1)}+2(|S_1\cup S_2|-1)\\
&\indent +\sum_{H_i\overline{H}_j\in \mathcal{E}}(2|(H_i\cap \overline{H}_j)\setminus (S_1\cap S_2)|-1)+2(|S_1\cap S_2|-1)\\
&=\sum_{i=1}^{r}{2(|V_i\cup S_1\cup S_2|-|S_1\cup S_2|)}+2(|S_1\cup S_2|-1)-r\\
&\indent +\sum_{H_i\overline{H}_j\in \mathcal{E}}{2(|H_i\cap \overline{H}_j|-|S_1\cap S_2|)}+2(|S_1\cap S_2|-1)-|\mathcal{E}|\\
&\leq \sum_{i=1}^{k}{2(|H_i|-|S_1|)}+\sum_{i=1}^{l}{2(|\overline{H}_i|-|S_2|)}\\
&\indent +2(|S_1\cup S_2|-1)+2(|S_1\cap S_2|-1)-k-l\\
&= \sum_{i=1}^{k}{2(|H_i|-|S_1|)}+\sum_{i=1}^{l}{2(|\overline{H}_i|-|S_2|)}+2|S_1|+2|S_2|-2-2-k-l\\
&= \sum_{i=1}^{k}{(2|H_i\setminus S_1|-1)}+2(|S_1|-1)+\sum_{i=1}^{l}{2|\overline{H}_i\setminus S_2|-1}+2(|S_2|-1),
\end{align*}
where the third inequality follows from (\ref{eq:rEkl1}) and (\ref{eq:double-count1}),
and the second last equality follows from the formula
$|S_1\cup S_2|+|S_1\cap S_2|=|S_1|+|S_2|$.
Therefore equality must hold throughout. Hence we can deduce that
$\Hc_{\text{union}}$ is $(S_1\cap S_2)$-tight (and $\Hc_{\text{int}}$ is $(S_1\cap S_2)$-tight).
\eproof

\subsection{\boldmath The matroid ${\mathcal M}_T(G)$}

Given a graph $G=(V,E)$ and a set $T\subseteq V$, we 
will show that the family
$
\mathcal{I}_T=\{I\subseteq E : G'=(V,I) \mbox{ is strongly $T$-sparse}\}
$
is the family of independent sets of a matroid ${\mathcal M}_T(G)$ on $E$ when $1\leq |T|\leq 3$. 

We need the following
concept.
For $S\subseteq V$ with $|S|\geq 2$, an {\em augmented $S$-compatible family} is a collection $\mathcal{L}=\{\Hc,X_1,\ldots,X_k\}$
where $\Hc$ is a (possibly empty) $S$-compatible family of subsets of $V$  and  $X_1,\ldots,X_k$ are subsets of $V$ of size at least two. 
We say that $\mathcal{L}$ {\em covers} an  edge $e\in  E$ if $e$ is induced by either a set $H\in \mathcal H$ or by one of the sets $X_i$ for some $1\leq i\leq k$ and let $i_G(\mathcal{L})$ be the number of edges of $G$ which are covered by $\mathcal{L}$. The family
$\mathcal{L}$ is $1$-{\em thin} if:
\begin{itemize}
%
%
%
\item[(T.1)] $|X_i\cap X_j|\leq 1$ for all pairs $1\leq i<j\leq k$;
\item[(T.2)]  $H_i\cap H_j=S$ for all distinct $H_i,H_j\in \Hc$; 
\item[(T.3)]  $|X_i\cap \bigcup_{H\in {\mathcal H}}H|\leq 1$ for all $1\leq i\leq k$.
\end{itemize}
We define the {\em $S$-value} of $\mathcal{L}$ to be
\[
\val_S(\mathcal{L})=
\left\{
\begin{array}{cl}
 \val_S(\Hc)+\sum_{i=1}^{k}(2|X_i|-3),&\text{if }\Hc\neq \emptyset\\[1mm]
\sum_{i=1}^{k}(2|X_i|-3), &\text{if }\Hc=\emptyset\,.
\end{array}
\right. 
\]
Note that, 
\begin{equation}\label{eq:thin}
\mbox{if $G$ is $S$-sparse and $\mathcal{L}$ is  $1$-thin,
then $i_G(\mathcal{L})\leq \val_S(\mathcal{L}).$}
\end{equation}

Recall that a family $\mathcal{I}$ of  subsets  of a finite set $E$ is the family of independent sets  in a matroid on $E$ if it satisfies the following 
three axioms:
\begin{itemize}
\item[(I.1)] $\emptyset\in \mathcal{I}$;
\item[(I.2)] if $I\in \mathcal{I}$ and $I'\subseteq I$, then $I'\in \mathcal{I}$;
\item[(I.3)] for all $E'\subseteq E$, every  maximal element of $\{I\in\mathcal I:I\subseteq E'\}$ 
has the same cardinality.
\end{itemize}

We will use these axioms to verify that $\mathcal{I}_T$ is the independent set system of a matroid and determine its rank function.  

\begin{theorem}\label{C3-thm:count_matroid}
Let $G=(V,E)$ be a graph, $T\subseteq V$ with $1\leq |T|\leq 3$, and $E_T$ be the set of edges of $G$ induced by $T$. Then
\begin{equation}\label{eq:matroid}
\mathcal{I}_T=\{I\subseteq E : G'=(V,I) \mbox{ is strongly $T$-sparse}\}
\end{equation}
is the family of independent sets in a matroid $\mathcal M_T(G)$ on $E$.
In addition, the rank of any $E'\subseteq E$ in $\mathcal M_T(G)$ is given by
$r(E')=\min\{\val_S(\mathcal{L})\}$ where the minimum is taken over all $S\subseteq T$ with $|S|\geq 2$ and all $1$-thin, augmented $S$-compatible families  
$\mathcal L$ which cover 
$E'\sm E_T$.
%
\end{theorem}
\bproof We show that $\mathcal{I}_T$ satisfies the independence axioms
(I.1), (I.2) and (I.3).
Since (I.1) and (I.2) follow immediately from the definition of $\mathcal{I}_T$, we only need to verify that  (I.3) holds. Choose $E'\subseteq E$.
We first show that
\begin{claim}\label{clm:rankF}
Let $F$ be a maximal element of $\{I\in{\mathcal I}_T:I\subseteq E'\}$.
Then $|F|=\val_S(\mathcal{L})$
for some $S\subseteq T$ with $|S|\geq 2$ and some $1$-thin, augmented $S$-compatible family  
$\mathcal L$ which covers 
$E'\sm E_T$.
\end{claim}
\begin{proof}[Proof of Claim]
Let $J=(V,F)$ denote the subgraph of $G$ induced by  $F$. Consider the following two cases.

\paragraph{\boldmath Case 1: $J$ has no tight $S$-compatible family for all $S\subseteq T$ with $|S|\geq 2$.} Let $X_1,X_2,\ldots,X_k$ be the maximal $T$-tight sets in $J$
and put
$
\mathcal{L}_1=\{\emptyset,X_1,X_2,\ldots,X_k\}
$.
Then $\mathcal{L}_1$ is an augmented $T$-compatible family.
Since $X=\{x,y\}$ 
is a $T$-tight set  in $J$ for all edges $xy\in F$, $\mathcal{L}_1$ covers $F$. In addition, Lemma \ref{C3-lem:tight_set_int} and the maximality of the sets $X_1,X_2,\ldots,X_k$ imply that $\mathcal{L}_1$ is 1-thin. Since each $X_i$ is $T$-tight in $J$ this gives,
\[
|F|=\sum_{i=1}^{k}{|E_J(X_i)|}=\sum_{i=1}^{k}{(2|X_i|-3)}=\val_T(\mathcal{L}_1).
\]
 We claim that $\mathcal{L}_1$ covers every edge of $E'\sm E_T$. To see this consider
an edge $ab\in E'\sm(F\cup E_T)$. Since $F$ is a maximal subset of $E'$ in $\mathcal{I}_T$
we have $F+ab\notin \mathcal{I}_T$. Our assumption that there is no $S$-tight,
$S$-compatible family in $J$, now implies that there is a $T$-tight set $X$ in $J$
with $a,b\in X$. Hence $X\subseteq X_i$ for some $1\leq i\leq k$. This implies that
$\mathcal{L}_1$ covers $ab$. Hence $\mathcal{L}_1$ covers every edge of $E'\sm E_T$ and the claim holds in this case.

\paragraph{\boldmath Case 2: $J$ has an $S$-tight, $S$-compatible family $\mathcal H$ for some $S\subseteq T$ with $|S|\geq 2$.} We may assume by Lemma \ref{C3-lem:max_cover} that, for every $S'$-tight, $S'$-compatible family ${\mathcal H}'$ in $J$ with $S'\subseteq T$ and $|S'|\geq 2$, we have $S'\subseteq S$ and  $\cov({\Hc}')\subseteq \cov(\Hc)$. Let $X_1,X_2,\ldots,X_k$ be the maximal $S$-tight sets of $J'=(V,F\sm E_J(\Hc))$ and put
$
\mathcal{L}_2=\{\Hc,X_1,X_2,\ldots,X_k\}.
$
Then $\mathcal{L}_2$ is an augmented $S$-compatible family which covers $F$ and we have 
\begin{equation}\label{eq:Xi}
\mbox{$i_{J'}(X_i)=2|X_i|-3$ for all $1\leq i\leq k$.}
\end{equation}

We next show that $\mathcal{L}_2$ is 1-thin.
Lemma \ref{C3-lem:tight_set_int} and the maximality of the sets $X_1,X_2,\ldots,X_k$ imply that $|X_i\cap X_j|\leq 1$ for all $1\leq i<j\leq k$, so (T.1) holds.
Lemma \ref{C3-lem:Hi-intersection} and the fact that $\Hc$ is $S$-tight imply that $H_i\cap H_j = S$ for all distinct $H_i,H_j\in \Hc$ (otherwise we could construct an $S$-compatible family $\overline \Hc$ in $J$ with $\val_S(\overline\Hc)<\val_S(\Hc)=i_J(\Hc)\leq i_J(\overline\Hc)$ and this would contradict the hypothesis that $F\in \mathcal{I}_T$). Hence (T.2) holds. Choose a set $X_i$. We show that (T.3) holds for $X_i$. If $|X_i\cap S|\geq 2$  then $\Hc'=\{X_i\}$ would be an $(S\cap X_i)$-tight, $(S\cap X_i)$-compatible family and the maximality of $\Hc$ would give $E_J(X_i)\subseteq \cov(\Hc')\subseteq \cov(\Hc)$. This  
would imply that $E_{J'}(X_i)=\emptyset$ and contradict (\ref{eq:Xi}).  Hence 
$|X_i\cap S|\leq 1$.

Suppose 
$|X_i\cap H_j|\geq 2$ for some $H_j\in \Hc$.
Then Lemma \ref{C3-lem:Hi_Y-int1} gives us an $S$-tight, $S$-compatible family $\overline \Hc$ with $\cov (H)\cup \cov (X_i)\subseteq \cov (\overline \Hc)$. The maximality of $\cov(\Hc)$ now implies that $\cov (X_i)\subseteq \cov (\Hc)$ and contradicts  (\ref{eq:Xi}). Hence $|X_i\cap H_j|\leq 1$ for all $H_j\in\Hc$. 

If $|X_i\cap S|= 1$ then the facts that $|X_i\cap H_j|\leq 1$ and $S\subset H_j$ for all $H_j\in \Hc$ would imply that (T.3) holds for $X_i$. So we may assume that $X_i\cap S= \emptyset$. We can now use Lemma \ref{C3-lem:Hi_Y-int1} and a similar argument to the previous paragraph to deduce that  $|X_i\cap H_j|=1$ for at most one $H_j\in \Hc$, so (T.3) holds for $X_i$.
%
%
%
%
%
%
%
%
Hence  $\mathcal{L}_2$ is
$1$-thin and we have
\begin{align*}
|F|&=\sum_{H_i\in\Hc}{|E_J(H_i)|}+\sum_{j=1}^{k}{|E_J(X_j)|}\\
&=\sum_{H_i\in\Hc}{(2|H_i\setminus S|-1)}+2(|S|-1)+\sum_{j=1}^{k}{(2|X_j|-3)}=\val_S(\mathcal{L}_2).
\end{align*}

We complete the proof of the claim by showing that $\mathcal{L}_2$ is a cover of $E'\sm E_T$. Choose $ab=e\in E'\sm (F\cup E_T)$.
By the maximality of $F$ we have $F+e\notin \mathcal{I}_T$. Thus $J$ has either
 an $S'$-tight set $X$ with $a,b\in X$ or  an $S'$-tight
$S'$-compatible family $\Hco$ with $a,b\in Y_i\in \Hco$ for some
$S'\subseteq T$ with $|S'|\geq 2$.
In the latter case, the maximality of $\cov(\Hc)$ implies that $\cov(\Hco)\subseteq \cov(\Hc)$
and hence $e$ is covered by $\mathcal{L}_2$. 
Hence we may assume that $a,b\in X$ for some $X\subseteq V$ with $i_J(X)=2|X|-3$.
%
If
$|X\cap\bigcup_{H_j\in \Hc}{H_j}|\geq 2$, then we can use a similar argument to that used to show that $X_i$ satisfies (T.3) to deduce that $\cov (X)\subseteq \cov (\Hc)$
which implies that 
$\mathcal{L}_2$ covers $e$. Hence we may assume that that
$|X\cap\bigcup_{H_j\in \Hc}{H_j}|\leq 1$. Then $E(X)\subseteq E(J')$ and hence
$X\subseteq X_i$ for some $1\leq i\leq k$.
Hence $e$ is covered by $\mathcal{K}_2$.
\end{proof}

We saw in (\ref{eq:thin}) that,  if $G$ is $S$-sparse and $\mathcal{L}$ is a $1$-thin, augmented $S$-compatible family
in $G$, then $i_G(\mathcal{L})\leq \val_S(\mathcal{L})$. Together with Claim \ref{clm:rankF}, this implies that $|F|=\min\{\val_S(\mathcal{L})\}$ where the minimum is taken over all $S\subseteq T$ with $|S|\geq 2$ and all $1$-thin, augmented $S$-compatible families  
$\mathcal L$ which cover 
$E'\sm E_T$. Since $\min\{\val_S(\mathcal{L})\}$ is independent of the choice of $F$, all maximal elements of $\{I\in{\mathcal I}_T:I\subseteq E'\}$ have the same cardinality. Hence (I.3) holds for $\mathcal{I}_T$ and $\mathcal M_T(G)$ is a matroid. The assertion that $r(E')=\min\{\val_S(\mathcal{L})\}$ follows immediately since $r(E')$ is equal to the cardinality of any maximal element of $\{I\in{\mathcal I}_T:I\subseteq E'\}$.
\eproof

The special cases of Theorem \ref{C3-thm:count_matroid} when $|T|=1,2$ are  given in \cite{LY} and \cite{FJK}, respectively.

\section{Coincident rigidity}\label{sec:main}
Throughout this section we will only be concerned with 2-dimensional frameworks so will suppress reference to the ambient space $\R^2$.
Let $G=(V,E)$ be a graph, $T\subseteq V$ with $|T|=3$. If $G$ has an independent, $T$-coincident realisation $(G,p)$ then $(G,p)$ is an independent, $S$-coincident realisation for all $S\subseteq T$ with $|S|=2$. We can combine this observation with Lemma \ref{C3-lem:geom_imp_count} to deduce that every independent set in the $T$-coincident rigidity matroid $\mathcal{R}_T(G)$ is independent in the matroid $\mathcal{M}_{T}(G)$ given by Theorem \ref{C3-thm:count_matroid}. Hence, to show that 
$\mathcal{R}_T(G)=\mathcal{M}_T(G)$, it only remains to show that every independent set in $\mathcal{M}_T(G)$ is independent in $\mathcal{R}_T(G)$. We will do this by induction on $|V|$: we suppose that $E$ is independent in $\mathcal{M}_T(G)$ and perform a graph theoretic reduction operation to create a smaller graph $G'=(V',E')$ such that $E'$ is independent in $\mathcal{M}_T(G')$; we apply induction to deduce that 
$E'$ is independent in $\mathcal{R}_T(G')$ then use the fact that the inverse of the reduction operation preserves independence in  the $T$-coincident rigidity matroid to deduce that $E$ is independent in $\mathcal{R}_T(G)$. The last step in this argument uses the following extension operations and geometric lemmas.  

Our first two lemmas concern the so called 0- and 1-extension  operations. We refer the reader to \cite{Whi1} for their proofs.
The {\em $0$-extension operation} on a graph $G=(V,E)$ constructs a new graph by adding a new vertex $w$ and two new edges from $w$ to $V$.
The {\em $1$-extension operation} constructs  a new graph from $G$ by deleting an edge $uv$ and then adding a new vertex $w$ and three new edges $wu, wv, wx$ for some $x\in V\sm \{u,v\}$.

\begin{lemma}\label{lem:0ext} 
Suppose that $G$ is a graph   and that $G'$ is obtained from $G$ by a  $0$-extension operation which  adds a new vertex $w$ and new edges $wu,wv$. Suppose further that $(G',p)$ is a realisation of $G'$ and that $u,v,w$ are not colinear in $(G,p)$.
Then $(G',p)$ is independent if and only if $(G,p|_{V(G)})$ is independent.
\end{lemma}

\begin{lemma}\label{lem:1ext} 
Suppose that $(G,p)$ is an independent framework   and that $G'$ is obtained from $G$ by a  $1$-extension operation which  adds a new vertex $w$. Suppose further that neighbours of $w$ in $G'$ are not colinear in $(G,p)$.
Then $(G',p')$ is independent for some $p'$ with $p'(x)=p(x)$ for all $x\in V(G)$.
\end{lemma}

Our third extension result is a geometric version of a generic vertex splitting lemma of Whiteley which is stated without proof  
in \cite{Whi1}. 
Given a graph $G=(V,E)$ and $v\in V$ with neighbour set  $N_G(v)$, the {\em vertex splitting operation}  chooses  pairwise disjoint sets  $U_1,U_2,U_3$ with $U_1\cup U_2\cup U_3=N_G(v)$ and $|U_2|=2$, deletes all edges from $v$ to $U_3$, and then adds a new vertex $v'$ and $|U_3|+2$ new edges from $v'$ to each vertex in $U_2\cup U_3$. 

\begin{lemma}\label{lem:split} 
    Suppose that $(G,p)$ is an independent framework   and that $G'$ is obtained from $G$ by a vertex splitting operation which splits a vertex $z\in V(G)$ into two vertices $z,z'$. Suppose further that the common neighbours $z_1,z_2$ of $z$ and $z'$ in $G'$ are not colinear with $z$ in $(G,p)$.  
Put $p'(z)=p'(z')=p(z)$ and $p'(x)=p(x)$ for all $x\in V(G)-z$. 
Then $(G',p')$ is independent.
\end{lemma}
\bproof 
Let $N_{G'}(z)=\{z_1,z_2,\ldots, z_k\}$ and $N_{G'}(z')=\{z_1,z_2,z_{k+1},z_{k+2},\ldots, z_m\}$.
Then the rigidity matrix 
$R(G',p')$ has the following form.
\renewcommand{\kbldelim}{[}
\renewcommand{\kbrdelim}{]}
\[
  \begin{blockarray}{ccc&cccc}
             &z            & z'           & z_1         & z_2         &        &  \\
	\begin{block}{c[cccccc]}
    zz_1     & p(z)-p(z_1) &  (0,0)      & p(z_1)-p(z) & (0,0)       &  \cdots&  \\
    zz_2     & p(z)-p(z_2) & (0,0)       &  (0,0)      & p(z_2)-p(z) &  \cdots&  \\
	z'z_1     &    (0,0)    & p(z)-p(z_1) & p(z_1)-p(z) & (0,0)       &  \cdots&  \\
	z'z_2     & (0,0)       & p(z)-p(z_2) & (0,0)       & p(z_2)-p(z) &  \cdots&  \\
	zz_3     & p(z)-p(z_3) & (0,0)       & (0,0      ) &             &  \cdots&  \\
    \vdots   & \vdots      & \vdots      & \vdots      & \vdots      &  \vdots&  \\
    zz_k     & p(z)-p(z_k) & (0,0)       & (0,0)       &  (0,0)      & \cdots &  \\
	z'z_{k+1} & (0,0) & p(z)-p(z_k)       & (0,0)       &  (0,0)      & \cdots &  \\
    \vdots   & \vdots      & \vdots      & \vdots      &   \vdots    & \vdots &  \\
	z'z_m     & (0,0) & p(z)-p(z_m)       & (0,0)       &   (0,0)     & \cdots &\\\cline{4-6}
	         &  0          &     \BAmulticolumn{1}{c|}{0}       &             &R(G-z,p)     &        &  \\
	\end{block}
	\end{blockarray}
\]
Let $M$ be the matrix obtained from $R(G',p')$ as follows: subtract row 3 from row 1 and row 4 from row 2, then add column 1 to column 2. Then {\footnotesize $M=\left(\begin{array}{cc}
p(z)-p(z_1) &  0\\
p(z)-p(z_2) &  0\\
* & R(G,p)
\end{array} \right)
$}. The hypotheses that $(G,p)$ is independent and $p(z),p(z_1),p(z_2)$ are not colinear now implies that $M$ has independent rows. Hence $(G',p')$ is independent.
%
%
%
%
%
\eproof

Our fourth extension lemma gives sufficient conditions for the operation of  replacing a rigid subgraph of a graph by  a larger rigid subgraph to preserve rigidity.

\begin{lemma}\label{lem:replace} Let $G=(V,E)$ be a graph, $Y\subseteq V$ such that $G[Y]$ is rigid and $\{Y_1,Y_2,\ldots,Y_m\}$ be a partition of $Y$ with $m\geq 3$. Let $G'=(V',E')$ be obtained from $G$ by contracting each set $Y_i$ to a single vertex $y_i$ for all $1\leq i\leq m$ and then adding an edge $y_iy_j$ for all non-adjacent pairs $y_i,y_j$, $1\leq i<j\leq m$. Put $Y'=\{y_1,y_2,\ldots,y_m\}$. Suppose that 
$(G',p')$ is an infintesimally rigid realisation of $G'$. Then $(G,p)$ is infinitesimally  rigid for some $p:V\to \R^2$ with
$p|_{V\sm Y}=p'|_{V'\sm Y'}$.  
\end{lemma}
\begin{proof}
Let $G^*$ be obtained from $G$ by adding all edges between the vertices of $Y$. Since $G[Y]$ is rigid, it will suffice to show that $(G^*,p)$ is infinitesimally  rigid for some $p:V\to \R^2$ with
$p|_{V\sm Y}=p'|_{V'\sm Y'}$.  
Let $p(v)=p'(v)$ for each $v\in V\sm Y$ and $p(v)=p'(y_i)$ for each $v\in Y_i$, $1\leq i\leq m$. Then $(G^*[Y],p|_Y)$ is infinitesimally rigid. In addition, each infinitesimal motion of $(G^*,p)$ which fixes $Y$ induces an infinitesimal motion of $(G',p')$ which fixes $Y'$. Since $(G',p')$ is infintesimally rigid, every such motion fixes $V'\sm Y'$. Hence $(G^*,p)$ is infinitesimally  rigid.
\end{proof}

Our final lemma complements Lemmas \ref{lem:0ext} and \ref{lem:1ext} by showing that the inverse of the 0- and 1-extension operations preserves the property of being strongly $T$-sparse.

\begin{lemma}\label{C3-lem:1-red}
Let $G=(V,E)$ be a graph and $T\subseteq V$ with $1\leq |T|\leq 3$. 
Suppose that $G$ is strongly $T$-sparse
and $z\in V\setminus T$ has at most one neighbour in $T$.\\
(a) If $d(z)=2$ then $G-z$ is strongly $T$-sparse.\\
(b)  If $d(z)=3$ then $G-z+xy$ is strongly $T$-sparse for some non-adjacent $x,y\in N_G(z)$.
\end{lemma}
\bproof Statement (a) follows immediately from the fact if $H$ is a subgraph of a strongly $T$-sparse graph and $T\subseteq V(H)$ then $H$ is strongly $T$-sparse. 

To verify (b) we  let $F=\{ab:a,b\in N(z)\}$, $G_1=G-z+F$ and $G_2=G+F$.
Let $r(H)$ denote the rank  of the edge set of an arbitrary subgraph $H\subseteq G_2$ in the matroid $\mathcal{M}_{T}(G_2)$ given by Theorem \ref{C3-thm:count_matroid}.
Suppose, for a contradiction that (b)  is false. Then $r(G_1)=r(G-z)$. Since $G$ is strongly $T$-sparse, $E$ is independent in  $\mathcal{M}_{T}(G_2)$ and hence  $r(G_1)=r(G-z)= r(G)-3$.
Choose a base $B_1$ of $\mathcal{M}_{T}(G_1)$ that contains  $F$ and extend it to a base $B_2$ of $\mathcal{M}_{T}(G_2)$.
Since the edges of $G_2[N(G(z)\cup \{z\}]\cong K_4$ is a circuit of $\mathcal{M}_{T}(G_2)$, 
we have $r(G_2)=|B_2|\leq |B_1|+2= r(G_1)+2$.
Hence $r(G)\leq r(G_2)\leq r(G)-1$, a contradiction.
\eproof

\begin{theorem}\label{C3-thm:count_imp_geom}
Let $G=(V,E)$ be a graph and $T\subseteq  V$ with $1\leq |T|\leq 3$. Then $E$ is independent in $\mathcal{R}_{T}(G)$ if and only if $G$ is strongly $T$-sparse.
\end{theorem}
\bproof 
 If $E$ is independent in $\mathcal{R}_{T}(G)$ 
then $G$ is strongly $T$-sparse by Lemma \ref{C3-lem:geom_imp_count}. 
Hence we need only verify the reverse implication. 
 Suppose for a contradiction that this is false and that $(G,T)$ has been chosen to be  a counterexample such that: $|V|$ as small as possible; subject to this condition, $|E|$ as large as possible; subject to these two conditions, the number of vertices of degree at most three in $G$ is as large as possible. 
It is easy to check that $E$ is independent in $\mathcal{R}_{T}(G)$ when $|V|\leq |T|+1$ so we may assume that $|V|\geq |T|+2$. Let $K$ be a complete graph with vertex set $V$. If $E$ is not a base of $\mathcal{M}_T(K)$ then we could add an edge of $K$ to $E$ to obtain a counterexample with more edges  than $G$. Hence $E$ is a base of  
$\mathcal{M}_T(K)$. This implies that $|E|=2|V|-3$. (We have 
 $|E|\leq 2|V|-3$ since $G$ is strongly $T$-sparse. On the other hand, $|E|\geq 2|V|-3$ since independence in $\mathcal{R}_T(K)$ implies independence in $\mathcal{M}_T(K)$ and we can construct an independent set of size $2|V|-3$ in   $\mathcal{R}_T(K)$ by starting with an edge joining two vertices of $V\sm T$ and then recursively applying Lemma \ref{lem:0ext}.)
 

\begin{claim}\label{clm:mindeg3} $G$ has minimum degree three.
\end{claim}
\begin{proof}[Proof of Claim]
Since 
$|E|= 2|V|-3$, $G$ has minimum degree at most three. Suppose, for a contradiction, that $G$ has a vertex $z$ with $d(z)\leq 2$. 

We first consider the case when $z\not\in T$. If $|N_G(z)\cap T|\leq 1$ then we can apply Lemma \ref{C3-lem:1-red}(a) to deduce that $G-z$ is strongly $T$-sparse, use the minimality of $V$ to deduce that every generic $T$-coincident framework $(G-z,p)$ is independent,
and then apply Lemma
\ref{lem:0ext} to  obtain an independent  $T$-coincident realsation of $G$.
This would imply that $E$ is independent in $\mathcal{R}_{T}(G)$ and contradict the choice of $G$. Hence $d(z)=2$ and $N_G(z)\subseteq T$. Let $\mathcal{H}=\{T\cup \{z\},V-z\}$. Then $\mathcal{H}$ is a $T$-compatible family with $\val_{T}(\Hc)=2|V|-4$. This contradicts the fact that $G$ is strongly $T$-sparse since $i_G(\Hc)=|E|=2|V|-3$.

Now suppose that $z\in T$. Then $N_G(z)\cap T=\emptyset$ since $G$ is strongly $T$-sparse. Let $T'=T-z$ if $|T|\geq 2$ and otherwise put $T'=\{z'\}$ where $z'$ is an arbitrary vertex in $V-z$. 
Then $G-z$ is strongly $T'$-sparse as it is a subgraph of $G$. The choice of $G$ now implies that  every generic $T'$-coincident framework $(G-z,p)$ is independent. We can now apply  Lemma \ref{lem:0ext} to obtain an independent  $T$-coincident realsation of $G$. This implies that $E$ is independent in $\mathcal{R}_{T}(G)$ and contradicts the choice of $G$.
\end{proof}

Let $W$ be the set of all vertices of $G$ having at least two neighbours in $T$
and put $X=T\cup W$.

\begin{claim}\label{clm:Xdeg3} Every  vertex of degree three in $G$ belongs to $X$.
\end{claim}
\begin{proof}[Proof of Claim] 
Suppose for a contradiction that there exists a vertex $z\in V\setminus X$ of degree three in $G$. Then we can apply Lemma \ref{C3-lem:1-red}(b) to deduce that $G-z+xy$ is strongly $T$-sparse for some non-adjacent $x,y\in N_G(z)$, use the minimality of $V$ to deduce that every generic $T$-coincident framework $(G-z+xy,p)$ is independent,
and then apply Lemma
\ref{lem:1ext} to  obtain an independent  $T$-coincident realsation of $G$.
This implies that $E$ is independent in $\mathcal{R}_{T}(G)$ and contradicts the choice of $G$.
\end{proof}

Let
$\Hc_0=\{T\cup \{w\}:w\in W\}$. Then $\Hc_0$ is a $T$-compatible family and $|W|+2|T|-2=\val_{T}(\Hc_0)\geq i_G(\Hc_0)\geq 2|W|$ since $G$ is $T$-sparse. Since $|E|=2|V|-3$, Claims \ref{clm:mindeg3}  and \ref{clm:Xdeg3} imply that $|X|=|T|+|W|\geq 6$. These two inequalities, combined with the fact that $|T|\leq 3$, give 

\begin{equation}\label{eq:equal}
\mbox{$|T|=3$, $\; 3\leq |W|\leq 4\:$ and $\;2|W|\leq i_G(\Hc_0)\leq |W|+2|T|-2=|W|+4$.}
\end{equation}
Since each vertex in $W$ has at least two neighbours in $T$,
(\ref{eq:equal}) and the hypothesis that $G$ is strongly $T$-sparse imply that we can label the vertices in $T$ as $u,v,w$ in such a way that: 
\begin{equation}
\mbox{$G[X]$ contains one of the graphs shown in Figure \ref{C3-fig:possibilities} as a spanning subgraph. 
}
\label{eq:equalfig}
\end{equation}

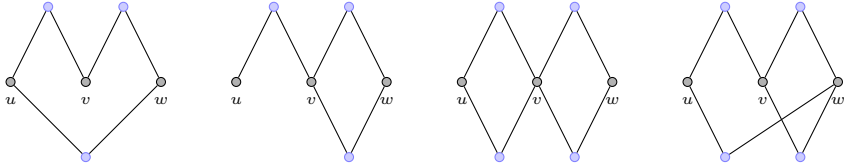
\begin{figure}[ht]
\begin{center}

\begin{tikzpicture}[scale=0.5,font=\tiny]

\node[nodeblack] at (-2cm,0cm) (u) [label=below:$u$] {};
\node[nodeblack] at (0cm,0cm) (v) [label=below:$v$] {};
\node[nodeblack] at (2cm,0cm) (w) [label=below:$w$] {};

\node[nodde] at (-1cm,2cm) (a2) [] {}
	edge[] (u)
	edge[] (v);
\node[nodde] at (1cm,2cm) (b2) [] {}
	edge[] (v)
	edge[] (w);
\node[nodde] at (0cm,-2cm) (c) [] {}
	edge[] (u)
	edge[] (w);
\begin{scope}[xshift=6cm]
\node[nodeblack] at (-2cm,0cm) (u) [label=below:$u$] {};
\node[nodeblack] at (0cm,0cm) (v) [label=below:$v$] {};
\node[nodeblack] at (2cm,0cm) (w) [label=below:$w$] {};

\node[nodde] at (-1cm,2cm) (a1) [] {}
	edge[] (u)
	edge[] (v);
\node[nodde] at (1cm,2cm) (a2) [] {}
	edge[] (v)
	edge[] (w);
\node[nodde] at (1cm,-2cm) (c) [] {}
	edge[] (v)
	edge[] (w);
\end{scope}
\begin{scope}[xshift=12cm]
\node[nodeblack] at (-2cm,0cm) (u) [label=below:$u$] {};
\node[nodeblack] at (0cm,0cm) (v) [label=below:$v$] {};
\node[nodeblack] at (2cm,0cm) (w) [label=below:$w$] {};

\node[nodde] at (-1cm,2cm) (a1) [] {}
	edge[] (u)
	edge[] (v);
\node[nodde] at (-1cm,-2cm) (e) [] {}
	edge[] (u)
	edge[] (v);
\node[nodde] at (1cm,2cm) (a2) [] {}
	edge[] (v)
	edge[] (w);
\node[nodde] at (1cm,-2cm) (c) [] {}
	edge[] (v)
	edge[] (w);
\end{scope}
\begin{scope}[xshift=18cm]
\node[nodeblack] at (-2cm,0cm) (u) [label=below:$u$] {};
\node[nodeblack] at (0cm,0cm) (v) [label=below:$v$] {};
\node[nodeblack] at (2cm,0cm) (w) [label=below:$w$] {};

\node[nodde] at (-1cm,2cm) (a1) [] {}
	edge[] (u)
	edge[] (v);
\node[nodde] at (-1cm,-2cm) (e) [] {}
	edge[] (u)
	edge[] (w);
\node[nodde] at (1cm,2cm) (a2) [] {}
	edge[] (v)
	edge[] (w);
\node[nodde] at (1cm,-2cm) (c) [] {}
	edge[] (v)
	edge[] (w);
\end{scope}
\end{tikzpicture}

\end{center}
\caption{Possible spanning subgraphs of $G[X]$ with $T=\{u,v,w\}$. 
}
\label{C3-fig:possibilities}
\end{figure}


We next use Lemma \ref{lem:split} to show that the first possibility in Figure \ref{C3-fig:possibilities} must occur.

\begin{claim}\label{clm:4cycle} $G[X]$ does not contain any of the three graphs on the right of Figure \ref{C3-fig:possibilities} as a spanning subgraph.
\end{claim}
\begin{proof}[Proof of Claim]
Suppose, for a contradiction, that the claim is false.
Then 
there exists a 4-cycle $C$ in $G[X]$ with $v,w\in V(C)$ and $E(C)\subseteq E(T,W)$. 

Consider the graph $G'=G/vw$.
Let $z$ be the vertex of $G'$ coresponding to $\{v,w\}$ and put $T'=\{u,z\}$.
If $G'$ is strongly $T'$-sparse
then the choice of $G$ implies that 
every generic $T'$-coincident framework $(G',p)$ is independent.
We could now apply Lemma \ref{lem:split} at $z$ to obtain an independent $T$-coincident realisation of $G$ and contradict the choice of $G$.
Hence $G'$ is not strongly $T'$-sparse.
Since $|T'|=2$, $G'$ is not  $T'$-sparse and hence there exists either 
a set $Z\subseteq V(G')$  such that 
$i_{G'}(Z)> 2|Z|-3$ or a $T'$-compatible family $\Hc$ in $G'$ such that 
$i_{G'}(\Hc)> \val_{T'}(\Hc)$.
%

Suppose the first alternative holds. We may assume that $Z$ has been chosen to be as small as possible. The minimality of $Z$ implies that
each vertex in $Z$ has degree at least three in $G'[Z]$, and the fact that $G$ is strongly $T$-sparse tells us
that $z\in Z$. If $u\in Z$ then $\Hc=\{Z\}$ would be a $T'$-compatible family with $i_{G'}(\Hc)> \val_{T'}(\Hc)$. We will consider this possibility when we investigate the second   alternative that $G'$ has a $T'$-compatible family $\Hc$  with $i_{G'}(\Hc)> \val_{T'}(\Hc)$, so we now assume that  $u\not\in Z$. This and the fact that $G[X]$ has one of the three graphs to the right of Figure \ref{C3-fig:possibilities} as a spanning subgraph imply that every vertex of $W\cap Z$ will have lower degree in $G'[Z]$ than in $G$. Since $G'[Z]$ has minimum degree at least three and $W$ has at most one vertex with degree geater than three in $G$ by (\ref{eq:equal}), we have $|W\cap Z|\leq 1$. 
We can now obtain a contradiction by examining the three choices for the spanning subgraphs of $G[X]$ given in  Figure \ref{C3-fig:possibilities}. 
If $G[X]$ contains the second graph in the figure as a spanning subgraph, we would have $d_{G'}(z)= d_G(v)+d_G(w)-2=4$ and the fact that 
$|W\cap Z|\leq 1$ now gives $d_{G'[Z]}(z)\leq 2$. On the other hand, if $G[X]$ contains the third or fourth  graph in the figure as a spanning subgraph, we would have $d_{G'}(z)= d_G(v)+d_G(w)-2\leq 5$ and the fact that 
$|W\cap Z|\leq 1$ again gives $d_{G'[Z]}(z)\leq 2$.
Both possibilities contradict the fact that $G'[Z]$ has minimum degree at least three.


Hence there exists a $T'$-compatible family $\Hc$ in $G'$ such that 
$i_{G'}(\Hc)> \val_{T'}(\Hc)$. We may assume that $\Hc$ has been chosen such that $i_{G'}(\Hc)$ is minimal. 
Let $C=vawbv$. 
Since at most one vertex in $X$
has degree greater than three, we may also assume that  $b$ has degree three in $G$.
Then $d_{G'}(b)=2$ and $T'\not \subseteq N_{G'}(b)$. We may now use the minimality of $i_{G'}(\Hc)$ to deduce that $b\not\in H$ for all $H\in \Hc$.
Consider the two  $T$-compatible families in $G$ given by
\begin{align*}
\Hc_1:=&\{H-z+v+w: H\in \Hc \}\cup \{\{a,u,v,w\},\{b,u,v,w\}\}
\end{align*}
and
\begin{align*}
\Hc_2:=&\{H-z+v+w: H\in \Hc \}\cup \{\{b,u,v,w\}\}.
\end{align*}
If $a\not\in H$ for all $H\in \Hc$ we have 
\[
\val_{T}(\Hc_1)=\val_{T'}(\Hc)+4<i_{G'}(\Hc)+4=i_G(\Hc_1)
\]
and, if $a\in H$ for some $H\in \Hc$, then
\[
\val_{T}(\Hc_2)=\val_{T'}(\Hc)+3<i_{G'}(\Hc)+3=i_G(\Hc_2).
\]
Both alternatives contradict the  the fact that $G$ is $T$-sparse.
%
%
\end{proof}

Our next result extends the previous claim by showing, in particular, that $G[X]$ is equal to the first graph of Figure \ref{C3-fig:possibilities}.  Let $Y=V\sm X$.

\begin{claim} \label{clm:6cycle} $G[X]$ is a cycle of length six, $|E_G(X,Y)|=6$,  $i_G(Y)=2|Y|-3$ and $G[Y]$ is rigid.
\end{claim}
\begin{proof}[Proof of Claim] 
Let $G_1$ be the first graph of Figure \ref{C3-fig:possibilities}.
Then $G_1$ is a spanning subgraph of $G[X]$ by  (\ref{eq:equalfig}) and Claim \ref{clm:4cycle}. 
If there exists an edge in $E_G(T,W)\setminus E(G_1)$ then $G[X]$ would contain the second graph in  Figure \ref{C3-fig:possibilities} as a spanning subgraph, contradicting Claim \ref{clm:4cycle}.
Thus $E_G(T,W)\subseteq E(G_1)$.
Claims \ref{clm:mindeg3}, the definition of $W$ and the fact that $E_G(T,W)\subseteq E(G_1)$ imply that each vertex in $T$ is adjacent to a distinct vertex in $Y$ and hence $|Y|\geq 3$.
Since $G$ is $T$-sparse, we have $i_G(Y)\leq  2|Y|-3$.
Let $\alpha$ be the number of edges in $G[X]-E(G_1)$. Then the fact that each vertex of $X$ has degree three in $G$ gives
\begin{align*}
2|V(G)|-3=|E(G)|=i_G(X)+|E_G(X,Y)|+i_G(Y)&\leq (6+\alpha)+(6-2\alpha)+2|Y|-3\\
&=2|V(G)|-3-\alpha.
\end{align*}
Hence $\alpha=0$ and equality holds throughout. This gives $G[X]=G_1$,
$|E_G(X,Y)|=6$ and $i_G(Y)= 2|Y|-3$. Since $G[Y]$ is $S$-sparse for any $S\subset Y$ with $|S|=1$, the minimality of $G$ implies that $G[Y]$ is rigid.
\end{proof}

%

We can now complete the proof of the theorem.
We first consider the case when $|Y|=3$.
The seven possibilities for $G$  when $|Y|=3$ are shown in Figure \ref{C3-fig:base_cases}. We can verify that the realisation $p:V\to \R^2$ shown for the first graph gives  an infinitesimally rigid, $T$-coincident realisation for all seven graphs by calculating the rank of the corresponding rigidity matrices. Hence the theorem holds when $|Y|=3$ so we must have $|Y|\geq 4$.

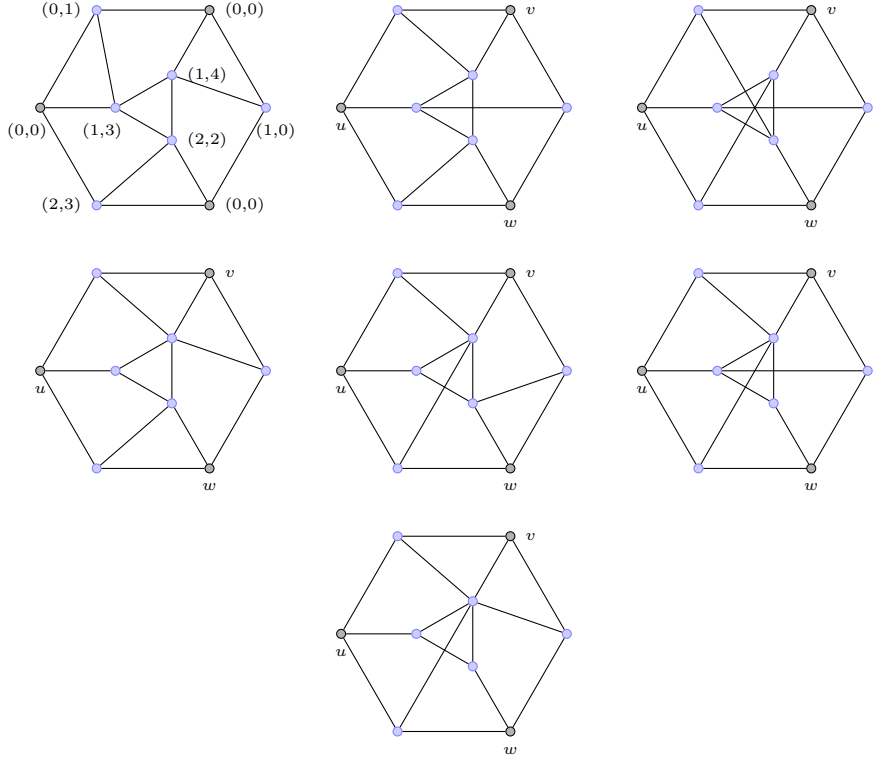
\begin{figure}[ht]
\begin{center}

\begin{tikzpicture}[scale=0.5,font=\tiny]

\node[nodeblack] at (180:3cm) (u) [label=below:\text{(0,0)\ \ \ \ }] {};
\node[nodeblack] at (60:3cm) (v) [label=right:\text{(0,0)}] {};
\node[nodeblack] at (300:3cm) (w) [label=right:\text{(0,0)}] {};

\node[nodde] at (120:3cm) (a) [label=left:\text{(0,1)}] {}
	edge[] (u)
	edge[] (v);
\node[nodde] at (0:3cm) (b) [label=below:\text{\ \ \ (1,0)}] {}
	edge[] (v)
	edge[] (w);
\node[nodde] at (240:3cm) (c) [label=left:\text{(2,3)}] {}
	edge[] (u)
	edge[] (w);

\node[nodde] at (180:1cm) (d) [label=below:\text{(1,3)\ \ \ \ }]{}
	edge[] (a)
	edge[] (u);
\node[nodde] at (60:1cm) (e) [label=right:\text{(1,4)}]{}
	edge[] (b)
	edge[] (v)
	edge[] (d);
\node[nodde] at (300:1cm) (f) [label=right:\text{(2,2)}]{}
	edge[] (c)
	edge[] (w)
	edge[] (d)
	edge[] (e);
	
\begin{scope}[xshift=8cm]
\node[nodeblack] at (180:3cm) (u) [label=below:$u$] {};
\node[nodeblack] at (60:3cm) (v) [label=right:$v$] {};
\node[nodeblack] at (300:3cm) (w) [label=below:$w$] {};

\node[nodde] at (120:3cm) (a) [] {}
	edge[] (u)
	edge[] (v);
\node[nodde] at (0:3cm) (b) [] {}
	edge[] (v)
	edge[] (w);
\node[nodde] at (240:3cm) (c) [] {}
	edge[] (u)
	edge[] (w);

\node[nodde] at (180:1cm) (d) []{}
	edge[] (u)
	edge[] (b);
\node[nodde] at (60:1cm) (e) []{}
	edge[] (v)
	edge[] (d)
	edge[] (a);
\node[nodde] at (300:1cm) (f) []{}
	edge[] (w)
	edge[] (d)
	edge[] (e)
	edge[] (c);
\end{scope}
\begin{scope}[xshift=16cm]
\node[nodeblack] at (180:3cm) (u) [label=below:$u$] {};
\node[nodeblack] at (60:3cm) (v) [label=right:$v$] {};
\node[nodeblack] at (300:3cm) (w) [label=below:$w$] {};

\node[nodde] at (120:3cm) (a) [] {}
	edge[] (u)
	edge[] (v);
\node[nodde] at (0:3cm) (b) [] {}
	edge[] (v)
	edge[] (w);
\node[nodde] at (240:3cm) (c) [] {}
	edge[] (u)
	edge[] (w);

\node[nodde] at (180:1cm) (d) []{}
	edge[] (u)
	edge[] (b);
\node[nodde] at (60:1cm) (e) []{}
	edge[] (v)
	edge[] (d)
	edge[] (c);
\node[nodde] at (300:1cm) (f) []{}
	edge[] (w)
	edge[] (d)
	edge[] (e)
	edge[] (a);
\end{scope}

\begin{scope}[yshift=-7cm]
\node[nodeblack] at (180:3cm) (u) [label=below:$u$] {};
\node[nodeblack] at (60:3cm) (v) [label=right:$v$] {};
\node[nodeblack] at (300:3cm) (w) [label=below:$w$] {};

\node[nodde] at (120:3cm) (a) [] {}
	edge[] (u)
	edge[] (v);
\node[nodde] at (0:3cm) (b) [] {}
	edge[] (v)
	edge[] (w);
\node[nodde] at (240:3cm) (c) [] {}
	edge[] (u)
	edge[] (w);

\node[nodde] at (180:1cm) (d) []{}
	edge[] (u);
\node[nodde] at (60:1cm) (e) []{}
	edge[] (v)
	edge[] (d)
	edge[] (a)
	edge[] (b);
\node[nodde] at (300:1cm) (f) []{}
	edge[] (w)
	edge[] (d)
	edge[] (e)
	edge[] (c);
\end{scope}

\begin{scope}[xshift=8cm,yshift=-7cm]
\node[nodeblack] at (180:3cm) (u) [label=below:$u$] {};
\node[nodeblack] at (60:3cm) (v) [label=right:$v$] {};
\node[nodeblack] at (300:3cm) (w) [label=below:$w$] {};

\node[nodde] at (120:3cm) (a) [] {}
	edge[] (u)
	edge[] (v);
\node[nodde] at (0:3cm) (b) [] {}
	edge[] (v)
	edge[] (w);
\node[nodde] at (240:3cm) (c) [] {}
	edge[] (u)
	edge[] (w);

\node[nodde] at (180:1cm) (d) []{}
	edge[] (u);
\node[nodde] at (60:1cm) (e) []{}
	edge[] (v)
	edge[] (d)
	edge[] (a)
	edge[] (c);
\node[nodde] at (300:1cm) (f) []{}
	edge[] (w)
	edge[] (d)
	edge[] (e)
	edge[] (b);
\end{scope}

\begin{scope}[xshift=16cm,yshift=-7cm]
\node[nodeblack] at (180:3cm) (u) [label=below:$u$] {};
\node[nodeblack] at (60:3cm) (v) [label=right:$v$] {};
\node[nodeblack] at (300:3cm) (w) [label=below:$w$] {};

\node[nodde] at (120:3cm) (a) [] {}
	edge[] (u)
	edge[] (v);
\node[nodde] at (0:3cm) (b) [] {}
	edge[] (v)
	edge[] (w);
\node[nodde] at (240:3cm) (c) [] {}
	edge[] (u)
	edge[] (w);

\node[nodde] at (180:1cm) (d) []{}
	edge[] (u)
	edge[] (b);
\node[nodde] at (60:1cm) (e) []{}
	edge[] (v)
	edge[] (d)
	edge[] (a)
	edge[] (c);
\node[nodde] at (300:1cm) (f) []{}
	edge[] (w)
	edge[] (d)
	edge[] (e);
\end{scope}

\begin{scope}[xshift=8cm,yshift=-14cm]
\node[nodeblack] at (180:3cm) (u) [label=below:$u$] {};
\node[nodeblack] at (60:3cm) (v) [label=right:$v$] {};
\node[nodeblack] at (300:3cm) (w) [label=below:$w$] {};

\node[nodde] at (120:3cm) (a) [] {}
	edge[] (u)
	edge[] (v);
\node[nodde] at (0:3cm) (b) [] {}
	edge[] (v)
	edge[] (w);
\node[nodde] at (240:3cm) (c) [] {}
	edge[] (u)
	edge[] (w);

\node[nodde] at (180:1cm) (d) []{}
	edge[] (u);
\node[nodde] at (60:1cm) (e) []{}
	edge[] (v)
	edge[] (d)
	edge[] (a)
	edge[] (b)
	edge[] (c);
\node[nodde] at (300:1cm) (f) []{}
	edge[] (w)
	edge[] (d)
	edge[] (e);
\end{scope}

\end{tikzpicture}

\end{center}
\caption{The possible graphs with $G[X]\cong C_6$ and $|Y|=3$. The vertices in $Y$ are drawn
inside the outer six-cycle which corresponds to $G[X]$.
}
\label{C3-fig:base_cases}
\end{figure}

Let $y_1,y_2,y_3$ be the vertices in $Y$ which are adjacent to $T$ and let $\{Y_1,Y_2,Y_3\}$ be a partition of $Y$ such that $y_i\in Y_i$ for all $1\leq i\leq 3$. 
Let $G'=(V',E')$ be obtained from $G$ by contracting each set $Y_i$ to a single vertex $y_i$ for all $1\leq i\leq 3$ and then adding an edge $y_iy_j$ for all non-adjacent pairs $y_i,y_j$ with $1\leq i<j\leq 3$. Put $Y'=\{y_1,y_2,y_3\}$. Then $G'$ is one of the graphs in  
Figure \ref{C3-fig:base_cases} so has an infinitesimally rigid $T$-coincident realisation $(G',p')$ by the previous paragraph. Then $(G,p)$ is infinitesimally  rigid for some $p:V\to \R^2$ with
$p|_{V\sm Y}=p'|_{V'\sm Y'}$ by Lemma \ref{lem:replace}. This implies that $(G,p)$ is independent since $|E|=2|V|-3$. This contradicts the choice of $G$ and completes the proof of the theorem. 
\end{proof}

\subsection{Proof of Theorem \ref{C3-thm:int_main1}}
Let $G=(V,E)$, $T=\{u,v,w\}$ and $E_T$ be the set of all edges of $G[T]$.

Necessity follows since, if $G$ has an an infinitesimally rigid
$T$-coincident realisation $(G,p)$, then  $(G',p)$ is infinitesimally rigid and (\ref{eq:dimKer})  implies that $(G'/S,p_S)$ is infinitesimally rigid
 for all $S\subseteq T$
with $|S|\geq 2$.

For sufficiency, we  assume that $G'$ and $G'/S$ are rigid  for all
$S\subseteq T$ with $|S|\geq 2$ and prove that $G$ is $T$-coincident rigid. Suppose, for a contradiction, that this is not the case. Then Theorems \ref{C3-thm:count_matroid} and \ref{C3-thm:count_imp_geom},
imply that, for some $S\subseteq T$ with $|S|\geq 2$, there exists a $1$-thin, augmented $S$-compatible family $\Lc=\{\Hc,X_1,X_2,\ldots,X_m\}$ in $G$ which covers $E-E_T$ and has $\val_S(\Lc)\leq 2|V|-4$. If $\Hc=\emptyset$ then we would have $\sum_{i=1}^m(2|X_i-3)\leq 2|V|-4$ and Theorem \ref{thm:LY} would imply that 
$G'$ is not rigid.
Hence 
$\Hc\neq \emptyset$.

%
%

Consider the graph $G'/S$ obtained from $G'$ by contracting the vertices in $S$  into a new vertex $z$. Then  $\mathcal{L}'=\{H_1',\ldots,H_k',X_1,\ldots,X_m\}$ is a 1-thin cover of $G/S$,
where $H_i'=(H_i\sm S)\cup \{z\}$ for each $H_i\in \Hc$. Then we have
\begin{align*}
\sum_{i=1}^{k}{(2|H_i'|-3)}+\sum_{i=1}^{m}{}(2|X_i|-3)
&=\sum_{i=1}^{k}(2|H_i\setminus S|-1)
+\sum_{i=1}^{m}(2|X_i|-3)
\\
&=\val_S(\mathcal{L})-2(|S|-1)\\
&\leq 2|V|-4-2(|S|-1)\\
&=2(|V|-(|S|-1))-4.
\end{align*}
This contradicts the assumption that  $G'/S$ is rigid by Theorem \ref{thm:LY}.
\qed

\medskip
The graph in Figure \ref{C3-fig:counter_example} shows that we cannot replace the hypothesis of Theorem \ref{C3-thm:int_main1} that $G'/S$ is rigid for all $S\subseteq T$ with $|S|\geq 2$ by the weaker hypothesis that  $G'/T$ is rigid. 
%
%

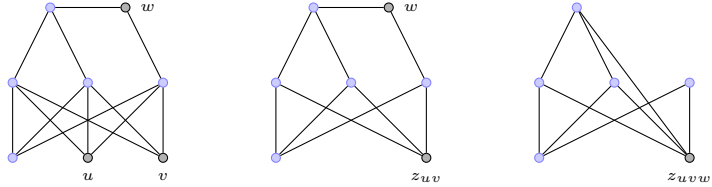
\begin{figure}[ht]
\begin{center}

\begin{tikzpicture}[scale=0.5,font=\tiny]

\node[nodde] at (-2cm,0cm) (a) [] {};
\node[nodeblack] at (0cm,0cm) (u) [label=below:$u$] {};
\node[nodeblack] at (2cm,0cm) (v) [label=below:$v$] {};
\node[nodde] at (-2cm,2cm) (b) [] {}
	edge[] (a)
	edge[] (u)
	edge[] (v);
\node[nodde] at (0cm,2cm) (c) [] {}
	edge[] (a)
	edge[] (u)
	edge[] (v);
\node[nodde] at (2cm,2cm) (d) [] {}
	edge[] (a)
	edge[] (u)
	edge[] (v);
\node[nodde] at (-1cm,4cm) (e) [] {}
	edge[] (b)
	edge[] (c);
\node[nodeblack] at (1cm,4cm) (w) [label=right:$w$] {}
	edge[] (e)
	edge[] (d);
	
\begin{scope}[xshift=7cm]
\node[nodde] at (-2cm,0cm) (a) [] {};
\node[nodeblack] at (2cm,0cm) (u) [label=below:$z_{uv}$] {};
\node[nodde] at (-2cm,2cm) (b) [] {}
	edge[] (a)
	edge[] (u);
\node[nodde] at (0cm,2cm) (c) [] {}
	edge[] (a)
	edge[] (u);
\node[nodde] at (2cm,2cm) (d) [] {}
	edge[] (a)
	edge[] (u);
\node[nodde] at (-1cm,4cm) (e) [] {}
	edge[] (b)
	edge[] (c);
\node[nodeblack] at (1cm,4cm) (w) [label=right:$w$] {}
	edge[] (e)
	edge[] (d);
\end{scope}

\begin{scope}[xshift=14cm]
\node[nodde] at (-2cm,0cm) (a) [] {};
\node[nodeblack] at (2cm,0cm) (u) [label=below:$z_{uvw}$] {};
\node[nodde] at (-2cm,2cm) (b) [] {}
	edge[] (a)
	edge[] (u);
\node[nodde] at (0cm,2cm) (c) [] {}
	edge[] (a)
	edge[] (u);
\node[nodde] at (2cm,2cm) (d) [] {}
	edge[] (a)
	edge[] (u);
\node[nodde] at (-1cm,4cm) (e) [] {}
	edge[] (b)
	edge[] (c)
	edge[] (u);
\end{scope}

\end{tikzpicture}

\end{center}
\caption{The graph on the left is $G$, $T=\{u,v,w\}$ and $G'=G$. The graph in the middle is
$G/uv$ and the graph on the right is $G/T$. Both $G$ and $G/T$ are rigid, but $G/uv$ is not. Hence $G$ is not $T$-coincident rigid
by Theorem \ref{C3-thm:int_main1}.}
\label{C3-fig:counter_example}
\end{figure}

\section{Closing Remarks}\label{sec:close}

\subsection{\boldmath Extension to $|T|\geq 3$} 
We believe that Theorem \ref{C3-thm:count_imp_geom} can be extended to sets $T$ of arbitrary size.

\begin{conj}\label{C3-con:count_imp_geom}
Let $G=(V,E)$ be a graph and $T$ be a non-empty subset of $V$. Then $E$ is independent in $\mathcal{R}_{T}(G)$ if and only if $G$ is strongly $T$-sparse.
\end{conj}
 
 The following result from the PhD thesis of the first author \cite{G} gives some evidence in support of this conjecture.
 
\begin{theorem}\label{thm:hakan}
Let $G=(V,E)$ be a graph, $T$ be a non-empty subset of $V$ and 
$$\mathcal{I}=\{I\subseteq E:G'=(V,I) \mbox{ is strongly $T$-sparse}\}.$$ Then $\mathcal{I}$ is the family of independent sets in a matroid on $E$.
\end{theorem}

\subsection{\boldmath Extension to $d\geq 3$}\label{subsec:K55}
It is natural to ask whether Theorems \ref{C3-thm:uv-FJK} and \ref{C3-thm:int_main1} can be extended to $\mathbb{R}^d$ for $d\geq 3$.
We will  use the following result on flexible realisations of complete bipartite graphs which follows easily from a result of Bolker and Roth \cite[Theorem 10]{BR}. It is stated explicitly in a paper of Whiteley as an immediate corollary of \cite[Theorem 1]{Whi}. Whiteley makes the simplifying assumption at the beginning of \cite{Whi} that all frameworks $(G,p)$ have $p(u)\neq p(v)$ whenever $uv\in E(G)$ but this assumption is not used in his proof of 
\cite[Theorem 1]{Whi}.

%
%
%
\begin{lemma}\label{C4-lem:nine_points_quadric}
Let $(K_{m,n},p)$ be a realisation of the complete bipartite graph $K_{m,n}$ with all its vertices on a quadric surface in $\R^d$ for some $m,n,d\geq 2$.
Then $(K_{m,n},p)$ is not infinitesimally rigid in $\R^d$.
\end{lemma}

Consider a generic $\{u,v\}$-coincident framework  $(K_{5,5},p)$ in $\R^3$ where  $u,v$ are vertices on different sides of the bipartition of $K_{5,5}$. Then Lemma \ref{C4-lem:nine_points_quadric}, combined with the fact that any set of nine points lie on a quadric surface
 in $\mathbb{R}^3$, imply that $(K_{5,5},p)$ is not infinitesimally rigid. On the other hand, $K_{5,5}-uv$  and $K_{5,5}/uv$ are both rigid in $\R^3$ since both $K_{5,5}-uv$ and the spanning subgraph of $K_{5,5}/uv$ obtained by deleting any three edges incident to the vertex of degree eight can be constructed from $K_4$ by the 3-dimensional versions of the 0- and 1-extension operations defined at the beginning of  Section \ref{sec:main}, see \cite{Whi1} for more details.

\paragraph{Acknowledgements} We would like to thank Shin-ichi Tanigawa for a helpful converstaion which gave rise to the $K_{5,5}$ example in Section \ref{subsec:K55}. The first author would also like to thank the Ministry of National Education of Turkey for PhD funding through a YLSY grant.

\end{document}